\pdfoutput=1
\documentclass[12pt,reqno]{amsart}

\usepackage[letterpaper,left=0.8in,right=0.79in,top=0.8in,bottom=0.8in]{geometry}

\usepackage[T1]{fontenc}
\usepackage[english]{babel}
\usepackage{lmodern}
\usepackage{microtype}

\usepackage{tikz-cd}

\usepackage{color}
\definecolor{darkgreen}{rgb}{0,0.45,0}
\usepackage[pdftex,bookmarks=false,colorlinks,urlcolor=blue,citecolor=darkgreen,linkcolor=darkgreen,linktocpage]{hyperref}

\usepackage[capitalise]{cleveref}
\usepackage{verbatim} %

\newtheorem{thm}{Theorem}[section]
\newtheorem{lem}[thm]{Lemma}
\crefname{lem}{Lemma}{Lemmas}
\newtheorem{prop}[thm]{Proposition}
\crefname{prop}{Proposition}{Propositions}
\newtheorem{cor}[thm]{Corollary}

\theoremstyle{definition}
\newtheorem{defn}[thm]{Definition}
\newtheorem{rem}[thm]{Remark}
\newtheorem{ex}[thm]{Example}

\newcommand{\lra}       {\longrightarrow}
\newcommand{\llra}[1]   {\stackrel{#1}{\lra}}  %

\mathchardef\usc="2D
\newcommand{\islocal}[1]{\mathsf{is\usc{}local}_{#1}}
\newcommand{\fib}{\mathsf{fib}}
\newcommand{\unit}{\mathbf{1}}
\newcommand{\susp}{\mathsf{\Sigma}}
\newcommand{\N}{\mathbb{N}}
\newcommand{\Z}{\mathbb{Z}}
\newcommand{\ap}{\mathsf{ap}}
\newcommand{\UU}{\mathcal{U}}
\newcommand{\Prop}{\mathsf{Prop}}
\newcommand{\refl}{\mathsf{refl}}
\newcommand{\id}{\mathsf{id}}
\newcommand{\pr}{\mathsf{pr}}
\newcommand{\tot}{\mathsf{tot}}
\newcommand{\defeq}{:\equiv}
\newcommand{\dfn}[1]{\textbf{\boldmath{#1}}}
\newcommand{\sbt}{{\,\begin{picture}(-1,1)(-1,-3)\circle*{3}\end{picture}\ } }
\newcommand{\pto}{\to_{\sbt}}
\newcommand{\IsEquiv}{\mathsf{IsEquiv}}

\newcommand{\cL}{\mathcal{L}}
\newcommand{\cR}{\mathcal{R}}
\newcommand{\sSet}{\textup{sSet}}
\newcommand{\Kan}{\textup{Kan}}
\newcommand{\LEM}{\textup{LEM}}
\newcommand{\AC}{\textup{AC}}

\DeclareMathOperator{\Hom}{Hom}

\begin{document}

\title{Non-accessible localizations}
\author{J. Daniel Christensen}
\email{jdc@uwo.ca}
\address{University of Western Ontario, London, Ontario, Canada}



\date{May 28, 2024}

\begin{abstract}
In a 2005 paper, Casacuberta, Scevenels and Smith construct a
homotopy idempotent functor $E$ on the category of simplicial sets
with the property that whether it can be expressed as localization
with respect to a map $f$ is independent of the ZFC axioms.
We show that this construction can be carried out in homotopy type theory.
More precisely, we give a general method of associating to a suitable
(possibly large) family of maps, a reflective subuniverse of any
universe $\UU$.
When specialized to an appropriate family, this produces a localization
which when interpreted in the $\infty$-topos of spaces agrees with the
localization corresponding to $E$.
Our approach generalizes the approach of~\cite{CSS} in two ways.
First, by working in homotopy type theory, our construction can be
interpreted in any $\infty$-topos.
Second, while the local objects produced by~\cite{CSS} are always
$1$-types, our construction can produce $n$-types, for any $n$.
This is new, even in the $\infty$-topos of spaces.
In addition, by making use of universes, our proof is very direct.

Along the way, we prove many results about ``small'' types that
are of independent interest.
As an application, we give a new proof that separated localizations exist.
We also give results that say when a localization with respect to
a family of maps can be presented as localization with
respect to a single map,
and show that the simplicial model satisfies a strong form of the
axiom of choice which implies that sets cover and that the law of
excluded middle holds.
\end{abstract}

\maketitle

\section{Introduction}

In topology, the study of reflective subcategories (usually called ``localizations'')
has a long history, beginning with work of Sullivan~\cite{Sul} and Bousfield~\cite{Bou}
for topological spaces.
This framework has played a fundamental organizing role in algebraic
topology in the intervening decades, influencing and leading to the solution
of many central conjectures in the field.
More recently, the theory of reflective subcategories was adapted to the
setting of $\infty$-categories by Lurie~\cite{Lur}.
The key ingredient of a localization $L$ is a function sending an object $X$
to an object $LX$, with certain properties.

Independently, logicians and philosophers have considered the notion
of modalities in logic, which allow one to qualify statements to
express \emph{possibility}, \emph{necessity}, and other attributes
such as \emph{temporal} modalities.
These modalities are expressed by applying a modal operator $\Diamond$ to
a proposition $P$ to produce a new proposition $\Diamond P$, with
certain properties.

Homotopy type theory is a framework that can be used to study $\infty$-toposes.
Thanks to recent work~\cite{dBB,dB,KL,LS,Shu}, we know that
anything proved in homotopy type theory is true in any $\infty$-topos%
\footnote{The initiality and semantics of higher inductive types still need to be fully worked out.}.
The types in homotopy type theory are interpreted as \emph{objects} of an $\infty$-topos,
and are also used to encode \emph{propositions} that one may wish to prove.
It turns out that the notion of a \emph{reflective subuniverse} in homotopy type theory
simultaneously encodes the idea of a localization of an $\infty$-topos%
\footnote{More precisely, a reflective subuniverse %
in homotopy type theory
corresponds to a family of reflective subcategories of each slice category,
compatible with pullback (see~\cite[Appendix~A]{RSS} and~\cite{Ver}).}
and common modalities studied in logic~\cite{Cor}.

The foundations for the theory of reflective subuniverses
in homotopy type theory were developed in~\cite{RSS}.
We briefly describe the background from there needed to explain our results.

\subsection*{Background}

All of our results are stated and proved within the framework of
homotopy type theory.
We follow the conventions and notation of~\cite{Uni}.
Most of the results have been formalized
using the Coq HoTT library~\cite{BGL+,HoTT}
and are available at~\cite{Chr}; see \cref{se:formalization} for details.

Let $\UU$ be a univalent universe.
Recall from \cite{RSS,Uni} that a \dfn{reflective subuniverse}
$L$ of $\UU$ consists of a subuniverse $\islocal{L} : \UU \to \Prop_{\UU}$,
a function $L : \UU \to \UU$,
and a \dfn{localization} $\eta_X : X \to LX$ for each $X : \UU$.
Being a localization means that $LX$ is $L$-local and $\eta_X$ is
initial among maps whose codomain is local.

Consider a family $f : \prod_{i:I} (A_i \to B_i)$ of maps.
A type $X$ is \dfn{$f$-local} if precomposition with $f_i$
\[
  X^{B_i} \to X^{A_i}
\]
is an equivalence for each $i : I$.
A key result in the subject is~\cite[Theorem~2.18]{RSS}, which states
that if the family $f$ is in $\UU$ in the sense that $I : \UU$ and all $A_i, B_i : \UU$,
then the subuniverse of $f$-local types in $\UU$ is reflective.
A reflective subuniverse that is presented in this way is called \dfn{accessible}.
(In classical topology,~\cite{Far} provides the foundations for such
accessible localizations.)

\subsection*{Motivation}

Given a reflective subuniverse $L$ of $\UU$, we can define a
new subuniverse $L'$ of $\UU$ consisting of the types whose
identity types are $L$-local.
In~\cite{CORS}, it is shown that $L'$ is again reflective.
In the proof, care was taken to avoid the assumption
that $L$ is accessible.

This raises the question of whether there are interesting
non-accessible reflective subuniverses.
In topology, work of~\cite{CSS} (building on~\cite{CGT,CRS}) shows that there exists a
reflective subcategory of the $\infty$-topos of spaces whose
accessibility is independent of ZFC.
In their example, every local space is a 1-type.
Their work is particularly interesting because it reveals a close connection to
large cardinal axioms.  They show that Vop\v{e}nka's principle implies that
all localizations are accessible, which in turn implies the existence of
measurable cardinals.
The only example I am aware of in type theory is the double-negation
modality, which is a reflection onto the subuniverse of stable
propositions (those $P$ such that $P \to \lnot \lnot P$ is an equivalence).
However, this is accessible if one assumes propositional resizing, which
holds in models associated to $\infty$-toposes~\cite[Proposition~11.3]{Shu}.

\subsection*{Main result}
The goal of the present work is to give a construction of a
reflective subuniverse in homotopy type theory which when interpreted
in spaces reproduces the example of~\cite{CSS} whose accessibility is
independent of ZFC.
Our result is in fact more general. First, it gives a result
in any model of homotopy type theory (e.g., in any $\infty$-topos).
Second, it also produces examples at higher truncation levels.  
Moreover, by working within homotopy type theory, we are able to give a simpler proof,
not at all following the classical proof.
In future work, we hope to give a proof within homotopy type theory
that our theorem can produce a reflective subuniverse that cannot
be proven to be accessible (or even merely accessible).

The main result is:

\theoremstyle{plain}
\newtheorem*{th:non-accessible}{Theorem \ref{th:non-accessible}}
\begin{th:non-accessible}
Let $n \geq -1$ and
let $f : \prod_{i:I} (A_i \to B_i)$ be a family of $(n-1)$-connected maps.
Then the subuniverse of $n$-truncated $f$-local types in $\UU$ is reflective.
\end{th:non-accessible}

The key point is that none of $I$, $A_i$ or $B_i$ are required to
be in $\UU$.  They can be in any other universe.
When $n = 1$, this theorem reproduces the classical result of~\cite{CSS}.
We explain this in \cref{se:pastwork}.
When $n > 1$, this result appears to be new, even for the $\infty$-topos
of spaces.
(We recall the definition of connectedness in \cref{de:connected}.)

In order to prove the main theorem, we first prove results that give
conditions implying that a type is equivalent to a type in $\UU$.
Say that a type $X$ is \dfn{small} if it is equivalent to a type in $\UU$.
More precisely, $X$ is \dfn{small} if $\sum_{X' : \UU} (X' \simeq X)$.
Note that this is a mere proposition, by univalence.
Define $X$ to be \dfn{$0$-locally small} if it is small.
Define $X$ to be \dfn{$(n+1)$-locally small} if for all $x, x' : X$, $x = x'$ is
$n$-locally small.
These are also mere propositions.

Using his join construction~\cite{R}, Rijke showed that if $A$ is small,
$X$ is 1-locally small and $A \to X$ is $(-1)$-connected (surjective on $\pi_0$),
then $X$ is small.
Rijke's result is the $n=0$ case of the following result:

\newtheorem*{pr:smallness}{Proposition \ref{pr:smallness}}
\begin{pr:smallness}
Let $n \geq -1$.
If $f : A \to X$ is $(n-1)$-connected, $A$ is small and $X$ is $(n+1)$-locally small,
then $X$ is small.
\end{pr:smallness}

To prove the main result, we consider the extended family $\bar{f}$,
indexed by $I + 1$, which also includes the map $S^{n+1} \to 1$.
The $\bar{f}$-local types are exactly the $n$-truncated types which
are $f$-local.
We show that localization with respect to $\bar{f}$ on a universe $\UU'$
that contains $\UU$, $I$ and all of the types $A_i$ and $B_i$ restricts
to a localization on $\UU$, using the proposition and~\cite[Theorem~3.12]{CORS}.
In addition, in order to adjust the predicate $\islocal{\bar{f}}$ to land in $\UU$,
we need to use \dfn{propositional resizing}, the axiom that every mere proposition
is small.

\subsection*{Organization}
In \cref{se:smallness}, we prove a number of results about small and
locally small types.  These include the proposition above, as well as
other results that are not needed for the main theorem, but may be of
use in other situations.
In \cref{se:non-accessible}, we prove the main theorem giving the
existence of localizations determined by a large family of maps.
We also show that the extended family $\bar{f}$ generates an
orthogonal factorization system on $\UU$.
The short \cref{se:L'} uses some of the earlier results to give
an alternative proof of the existence of $L'$ localizations.
Then in \cref{se:single-map}, we show that under certain assumptions,
an accessible localization can be presented as a localization with respect
to a single map.
We then show that a strong form of the axiom of choice holds in the
simplicial model, which implies the needed assumptions.
In \cref{se:pastwork}, we describe the relationship between our
results and those of~\cite{CSS}.
In \cref{se:formalization}, we briefly describe the formalization
of these results, which is available at~\cite{Chr}.

\subsection*{Acknowledgements}
The author thanks Mathieu Anel, Ulrik Buchholtz, Jonas Frey, Egbert Rijke,
Mike Shulman and Andrew Swan for comments that helped to improve this paper.

\section{Smallness}\label{se:smallness}

In this section, we prove a number of results about small types.
The only results that will be used in the rest of the paper are
\cref{pr:smallness} and \cref{le:truncated-type,le:truncated-family}.
 
We defined ``small'' and ``$n$-locally small'' in the introduction.
Note that these collections are $\Sigma$-closed:  if $Y$ is $n$-locally
small and $A : Y \to \UU'$ is a family of $n$-locally small types in some
universe $\UU'$, then $\sum_{y:Y} A(y)$ is $n$-locally small.
These collections are also clearly closed under equivalence.

We will make frequent use of the concept of connectedness.

\begin{defn}\label{de:connected}
Let $k \geq -2$.
A type $A$ is \dfn{$k$-connected} if its $k$-truncation $\|A\|_k$ is contractible.
A map $f : A \to B$ is \dfn{$k$-connected} if for every $b : B$, the
fibre $\fib_f(b) \defeq \sum_{a : A} \, (f(a) = b)$ is $k$-connected.
\end{defn}

For example, the $(-1)$-connected maps are precisely the surjections,
i.e., the maps that are surjective on $\pi_0$.

Using his join construction~\cite{R}, Rijke showed that if $A$ is small,
$X$ is 1-locally small and $A \to X$ is $(-1)$-connected, then $X$ is small.
We generalize this.

\begin{prop}\label{pr:smallness}
Let $n \geq -1$.
If $f : A \to X$ is $(n-1)$-connected, $A$ is small and $X$ is $(n+1)$-locally small,
then $X$ is small.
\end{prop}

\begin{proof}
When $n = -1$, this is vacuous, since $0$-locally small means the same as small.

Now let $n \equiv m + 1$ for $m \geq -1$.
Assume $f : A \to X$ is $(n-1)$-connected, $A$ is small and $X$ is $(n+1)$-locally small.
We'll show below that $X$ is 1-locally small.
Then since $(n-1)$-connected implies $(-1)$-connected, the inductive step will
follow from the join construction.

Let $x, x' : X$.  We need to show that $x = x'$ is small.
Since being small is a mere proposition and $f$ is surjective, we can assume that $x \equiv f(a)$
and $x' \equiv f(a')$ for some $a, a' : A$.
We'll show that $f(a) = f(a')$ is small.
We have $\ap f : a = a' \to f(a) = f(a')$.
Note that $a = a'$ is small and this map is $(n-2)$-connected.
By assumption, $f(a) = f(a')$ is $n$-locally small.
So, by the induction hypothesis, $f(a) = f(a')$ is small.
\end{proof}

Next we give some results that will be used to prove a dual version
of~\cref{pr:smallness}.

\begin{lem}\label{le:truncated-type}
Let $n \geq -1$.
If $X$ is $n$-truncated, then $X$ is $(n+1)$-locally small.
\end{lem}

\begin{proof}
We prove this by induction on $n$.
When $n = -1$, this is a restatement of propositional resizing, which we are assuming holds.

Now let $n \equiv m + 1$ with $m \geq -1$, and assume that $X$ is $(m+1)$-truncated.
This means that for each $x, y : X$, the identity type $x = y$ is $m$-truncated.
By the induction hypothesis applied to $x = y$, we see that each $x = y$ is $m$-locally small.
That is, $X$ is $(m+1)$-locally small, as required.
\end{proof}

\begin{lem}\label{le:truncated-family}
Let $n \geq -1$.
If $f : X \to Y$ is $n$-truncated and $Y$ is $(n+1)$-locally small,
then $X$ is $(n+1)$-locally small.
\end{lem}

Recall that a map is said to be $n$-truncated if all of its fibres are $n$-truncated.
For example, a $(-1)$-truncated map is exactly an embedding, so the $n=-1$
case of the lemma says that every subtype of a small type is small.

\begin{proof}
Note that $X$ is equivalent to $\sum_{y:Y} \fib_f(y)$.
The fibres are $n$-truncated by assumption, so by \cref{le:truncated-type} they
are $(n+1)$-locally small.
The claim follows from the fact that the $(n+1)$-locally small types are $\Sigma$-closed.
\end{proof}

\begin{lem}\label{le:dual-smallness}
Let $n \geq -1$.
If $f : X \to Y$ is $n$-truncated, $Y$ is $(n+1)$-locally small and
$X$ is $n$-connected, then $X$ is small.
\end{lem}

\begin{proof}
By \cref{le:truncated-family}, $X$ is $(n+1)$-locally small.
Since $\|X\|_n \simeq 1$ and $n \geq -1$ we have that $\|X\|_{-1} \simeq 1$.
Since we are proving a mere proposition, we can assume that we have $x : X$.
Consider the map $1 \to X$ that selects the point $x$.
Since $X$ is $n$-connected, this map is $(n-1)$-connected~\cite[Lemma~7.5.11]{Uni}.
So by \cref{pr:smallness}, $X$ is small.
\end{proof}

This allows us to characterize small types using truncations.

\begin{thm}\label{th:characterize-smallness}
Let $n \geq -1$ and let $X$ be a type.  Then $X$ is small if and only if
$X$ is $(n+1)$-locally small and $\|X\|_n$ is small.
\end{thm}

\begin{proof}
If $X$ is small, then clearly it is $(n+1)$-locally small and its $n$-truncation is small.

To prove the converse, consider the truncation map $|{-}| : X \to \|X\|_n$ and note that
\[
  X \simeq \sum_{y:\|X\|_n} \, \sum_{x:X} \, (|x| = y) .
\]
Thus it suffices to show that the fibre $\sum_{x:X} \, (|x| = y)$ is small for each
$y : \|X\|_n$.
The first projection map $\sum_{x:X} \, (|x| = y) \to X$ is $(n-1)$-truncated,
since the types $|x| = y$ are $(n-1)$-truncated.
In particular, this map is $n$-truncated.
Moreover, $\sum_{x:X} \, (|x| = y)$ is $n$-connected~\cite[Corollary 7.5.8]{Uni}.
Therefore, by \cref{le:dual-smallness}, $\sum_{x:X} \, (|x| = y)$ is small,
as required.
\end{proof}

We also give an alternate proof of the converse that avoids \cref{le:dual-smallness},
thus showing that \cref{th:characterize-smallness} doesn't require propositional resizing.%
\footnote{This is relevant for work in progress of Bezem, Coquand, Dybjer and Escardó.}

\begin{proof}[Proof 2]
As above, it suffices to show that the fibre $\sum_{x:X} \, (|x| = y)$
of the $n$-truncation map is small for each $y : \|X\|_n$.
Since being small is a proposition, we can assume that
$y \equiv |x'|$ for some $x' : X$.
We apply \cref{pr:smallness} with $f : \unit \to \sum_{x:X} \, (|x| = |x'|)$
the constant map with value the pair $(x', \refl)$.
The codomain of $f$ is $n$-connected, as it is a fibre of the $n$-truncation map,
so $f$ is $(n-1)$-connected.
So it remains to show that the codomain of $f$ is $(n+1)$-locally small.
We are assuming that $X$ is $(n+1)$-locally small, so it suffices to show
that for each $x : X$, $|x| = |x'|$ is $(n+1)$-locally small.
This type is equivalent to $|x = x'|_{n-1}$, where the subscript indicates
the shift in truncation level here.
We could apply \cref{le:truncated-type} now, but that uses propositional resizing.
However, one can see that it suffices to show that $x = x'$ is $(n+1)$-locally small,
and this follows from the assumption that $X$ is $(n+1)$-locally small (since that
implies $X$ is $(n+2)$-locally small).
\end{proof}

We can now give our dual version of \cref{pr:smallness},
which generalizes \cref{le:dual-smallness}.

\begin{cor}\label{co:dual-smallness}
Let $n \geq -1$.
If $f : X \to Y$ is $n$-truncated, $Y$ is $(n+1)$-locally small and
$\|X\|_n$ is small, then $X$ is small.
\end{cor}

\begin{proof}
By \cref{le:truncated-family}, $X$ is $(n+1)$-locally small.
So by \cref{th:characterize-smallness}, $X$ is small.
\end{proof}

The next result can be used to replace the condition that $\|X\|_n$ be
small in \cref{th:characterize-smallness} and \cref{co:dual-smallness}
with a condition on homotopy groups.

\begin{prop}\label{pr:homotopy}
Let $n \geq 0$ and let $X$ be an $n$-truncated type.
Then $X$ is small if and only if
$\pi_0(X)$ is small and
for every $i : \N$ and $x : X$, $\pi_i(X, x)$ is small.
\end{prop}

The assumption on $\pi_i(X, x)$ is redundant when $i > n$, since
$X$ is $n$-truncated and contractible types are always small.
Note that we separately assume that $\pi_0(X)$ is small, since
$X$ might not have any global elements (but see \cref{re:issmall-inhabited-issmall}).

\begin{proof}
The ``only if'' direction is clear.

We prove the converse by induction on $n$.
When $n = 0$, $X$ is equivalent to $\pi_0(X)$, which is assumed to be small.

Now assume that $n \equiv m+1$ for $m : \N$.
As in \cref{th:characterize-smallness},
consider the truncation map $|{-}| : X \to \|X\|_m$ and note that
\[
  X \simeq \sum_{y:\|X\|_m} \, \sum_{x:X} \, (|x| = y) .
\]
By the inductive hypothesis, $\|X\|_m$ is small.
Again, we take $y \equiv |x'|$ for some $x' : X$, and need to show
that the fibre $F \defeq \sum_{x:X} \, (|x| = |x'|)$ is small.
This type is pointed, $m$-connected and $(m+1)$-truncated, so it is
an Eilenberg--Mac Lane space for some group $G$, by~\cite[Theorem~5.1]{BDR}.
From the long exact sequence of the fibre sequence, we see that $G \simeq \pi_{m+1}(X,x')$,
which we have assumed is small.
The construction of Eilenberg--Mac Lane spaces~\cite{LF}
then implies that the fibre $F$ is small, as required.
\end{proof}

\begin{rem}\label{re:issmall-inhabited-issmall}
One can show that if $X$ being inhabited implies that $X$ is small,
then $X$ is small.  In other words, if we have a function
\[
  X \to \left(\sum_{X' : \UU} (X' \simeq X)\right) ,
\]
then $X$ is small.
This uses propositional resizing and is in fact equivalent to it.
Since this has been formalized and we are not using it, we omit the proof.
Because of this result, one can remove ``$\pi_0(X)$ is small and'' from the statement
of \cref{pr:homotopy}, giving a slightly cleaner result.
\end{rem}

\begin{ex}
Let $G$ and $H$ be groups, and write $BG \equiv K(G,1)$ and $BH \equiv K(H,1)$
for their classifying spaces~\cite{LF}.
Suppose that $f : BH \pto BG$ is a $0$-truncated, pointed map.
This means that its fibres are sets, or equivalently that the corresponding
map $H \to G$ is a subgroup inclusion.
If $BG$ is small, then \cref{le:dual-smallness} implies that $BH$ is
small, since it is $0$-connected.
\end{ex}

\begin{ex}\label{ex:surjective}
With the same notation,
suppose that $f : BG \pto BH$ is a $0$-connected, pointed map, which means that it
represents a surjection of groups~\cite[Corollary~8.8.5]{Uni}.
Suppose that $BG$ is small.  Since $BH$ is a 1-type, it is 2-locally small.
So by \cref{pr:smallness}, $BH$ is small.
\end{ex}

\section{Non-accessible localizations}\label{se:non-accessible}

This section contains the main results of the paper.

We will make use of the concept of an orthogonal factorization system.
The following definition is equivalent to the definition in~\cite{RSS}.

\begin{defn}\label{de:OFS}
An \dfn{orthogonal factorization system on $\UU$} is a pair of predicates
$\cL,\, \cR : \prod_{X, Y : \UU} (X \to Y) \to \Prop_{\UU}$
such that $\cL = {}^{\perp}\cR$, $\cR = \cL^{\perp}$, and every
map $g$ in $\UU$ can be factored as $g = k \circ h$, with $h$ in $\cL$
and $k$ in $\cR$.
Here ${}^{\perp}\cR$ denotes the class of maps with the unique left
lifting property with respect to $\cR$,
and $\cL^{\perp}$ denotes the class of maps with the unique right
lifting property with respect to $\cL$.
\end{defn}

A key family of examples has as left class the $n$-connected maps and
as right class the $n$-truncated maps, for each $n \geq -2$.
(See~\cite[Section~7.6]{Uni} or~\cite[Theorem~1.34]{RSS}.)

We begin with an easy result that we will use to prove the main theorem.

\begin{prop}\label{pr:restrict}
Let $\UU$ be a universe contained in a universe $\UU'$.
Let $L$ be a reflective subuniverse of $\UU'$ such that for every $X : \UU$,
$LX$ is small.
Then the subuniverse of $L$-local types in $\UU$ is reflective.
\end{prop}

\begin{proof}
By propositional resizing, we can replace the mere proposition
$\islocal{L} : \UU' \to \Prop_{\UU'}$ defining the subuniverse $L$
with an equivalent one landing in $\Prop_{\UU}$.
Write $\islocal{\tilde{L}} : \UU \to \Prop_{\UU}$ for the restriction of
the new predicate to $\UU$, giving us a subuniverse $\tilde{L}$ of $\UU$.

We are given $\mathsf{is} : \prod_{X : \UU} \sum_{X' : \UU} (X' \simeq LX)$.
For $X : \UU$, define $\tilde{L} X$ to be $\pr_1 (\mathsf{is}\, X)$, so that
$\tilde{L} X : \UU$ and $\tilde{L} X \simeq LX$.
In particular, $\tilde{L} X$ is in the subuniverse $\tilde{L}$.

Define $\tilde{\eta} : X \to \tilde{L} X$ to be the composite
$X \to LX \to \tilde{L} X$, where the first map is the localization
in $\UU'$ and the second map is the inverse of the given equivalence.
It is straightforward to check that $\tilde{\eta}$ is a localization.
\end{proof}

Now we give the main theorem.

\begin{thm}\label{th:non-accessible}
Let $n \geq -1$ and
let $f : \prod_{i:I} (A_i \to B_i)$ be a family of $(n-1)$-connected maps,
with no smallness hypotheses on $I$, $A_i$ or $B_i$.
Then the subuniverse of $n$-truncated $f$-local types in $\UU$ is reflective.
\end{thm}

\begin{proof}
Recall that a type is $n$-truncated if and only if it is local
with respect to the single map $S^{n+1} \to 1$.
Motivated by this,
consider the extended family $\bar{f}$, indexed by $I + 1$, which
is defined to be $f$ on the left summand and which sends elements of the right summand
to the map $S^{n+1} \to 1$.
Then a type is $\bar{f}$-local if and only if it is $f$-local and $n$-truncated.

Let $\UU'$ be a universe containing $\UU$, $I$ and all of the types $A_i$ and $B_i$.
By~\cite[Theorem~2.18]{RSS}, the subuniverse of $\bar{f}$-local types in $\UU'$
is reflective.
By \cref{pr:restrict}, it is enough to
show that for $X : \UU$, $L_{\bar{f}}X$ is small.
By \cref{le:truncated-type}, $L_{\bar{f}} X$ is $(n+1)$-locally small.
Now, since all of the maps in the family $\bar{f}$ are $(n-1)$-connected
(including the additional map $S^{n+1} \to 1$),
and the $(n-1)$-connected maps are the left class in an orthogonal factorization system,
\cite[Theorem~3.12]{CORS} implies that the map
$\eta : X \to L_{\bar{f}}X$ is $(n-1)$-connected.
Therefore, by \cref{pr:smallness}, $L_{\bar{f}} X$ is small.
\end{proof}

The case $n=1$ recovers the classical result of~\cite[Section~6]{CSS}, whose local objects
are all 1-types.
See \cref{se:pastwork} for more details.

\begin{rem}\label{re:local-types}
As mentioned in the proof of \cref{th:non-accessible}, for each $X$, the
type $L_{\bar{f}} X$ is $n$-truncated and the map $\eta_{\bar{f}} : X \to L_{\bar{f}} X$
is $(n-1)$-connected.  It follows that for each $x : X$, the induced map
\[
  \pi_i(X,x) \lra \pi_i(L_{\bar{f}} X,\, \eta_{\bar{f}} x)
\]
is an equivalence for $i < n$ and surjective for $i = n$, while for $i > n$,
the codomain is trivial.
As a consequence, if $X$ is an $(n-1)$-type, then
$\eta_{\bar{f}} : X \to L_{\bar{f}} X$ is an equivalence, by the truncated
Whitehead theorem.  In other words, every $(n-1)$-type is $\bar{f}$-local.
So the $\bar{f}$-local types sit between the $(n-1)$-types and the $n$-types.
\end{rem}

\begin{rem}
For a type $X$, consider the composite $X \to \|X\|_n \to L_f \|X\|_n$ of
the localization maps.  Just as for $\eta_{\bar{f}}$, both of these maps
are $(n-1)$-connected.  In particular, $L_f \|X\|_n$ is $n$-truncated
and is therefore $\bar{f}$-local.
Each of the localization maps has the property that it is sent to an
equivalence by $Z^{(-)}$ for any $\bar{f}$-local type $Z$.
Indeed, the first is initial among maps into all $n$-truncated types, and the
second is initial among maps into all $f$-local types.
Therefore, this composite also presents the $\bar{f}$-localization of $X$.
That is, there is an equivalence $e : L_{\bar{f}} X \to L_f \|X\|_n$ making the
diagram
\[
  \begin{tikzcd}
    X \arrow[r,"|-|"] \arrow[d,swap,"\eta_{\bar{f}}"] & \|X\|_n \arrow[d,"\eta_f"] \\
    L_{\bar{f}} X \arrow[r,"\sim","e"'] & L_f \|X\|_n
  \end{tikzcd}
\]
commute.
\end{rem}

Under the same hypotheses as \cref{th:non-accessible}, we can get an
orthogonal factorization system on $\UU$.
This possibility was suggested to me by Mathieu Anel.

\begin{thm}\label{th:OFS}
Let $n \geq -1$ and
let $f : \prod_{i:I} (A_i \to B_i)$ be a family of $(n-1)$-connected maps.
Let $\cR$ consist of those maps in $\UU$ which are $n$-truncated and right orthogonal to $f$,
and let $\cL$ consist of those maps in $\UU$ which are left orthogonal to $\cR$.
Then $(\cL, \cR)$ is an orthogonal factorization system on $\UU$.
\end{thm}

\begin{proof}
The proof parallels the proof of \cref{th:non-accessible}.
Recall that a map is $n$-truncated if and only if it is right orthogonal to the
single map $S^{n+1} \to 1$.
Extend $f$ to a family $\bar{f}$ of maps indexed by $I + 1$, which also includes
the map $S^{n+1} \to 1$.
Then a map is right orthogonal to $\bar{f}$ if and only if it is $n$-truncated
and right orthogonal to $f$.

Let $\UU'$ be a universe containing $\UU$, $I$ and all of the types $A_i$ and $B_i$
in the family $f$.
By~\cite[Theorem~2.41]{RSS}, there is an orthogonal factorization system
$(\cL', \cR')$ on $\UU'$, with $\cR' = \bar{f}^{\perp}$ and $\cL' = {}^{\perp}\cR'$.
Let $g : X \to Y$ be a map in $\UU$, and consider an $(\cL',\cR')$-factorization
\[
  X \llra{l} W \llra{r} Y
\]
of $g$.
We will show below that $W$ is small.
Therefore, after composing with appropriate equivalences, we can assume that $W$ is in $\UU$.
It follows that $r$ is in $\cR$, since $\cR$ consists precisely of those maps in
$\cR'$ which are in $\UU$.
The map $l$ is left orthogonal to $\cR'$, so it is also left orthogonal to $\cR$,
and hence is in $\cL$.

We now show that $W$ is small.
Since $r$ is in $\cR'$, it is $n$-truncated.
Therefore, by \cref{le:truncated-family}, $W$ is $(n+1)$-locally small.
Since every map in $\bar{f}$ is $(n-1)$-connected, $\cR'$ contains the $(n-1)$-truncated maps.
Therefore, $\cL'$ is contained in the $(n-1)$-connected maps,
and so $l$ is $(n-1)$-connected.
\cref{pr:smallness} then implies that $W$ is small.
\end{proof}

\cref{th:non-accessible} follows from \cref{th:OFS} by considering the
factorization of the map $X \to 1$ for each $X$ in $\UU$.

\section{An application to separated localizations}\label{se:L'}

In this short section, we give an alternate proof of a result from~\cite{CORS}.
This is not needed in the remainder of the paper.

We first recall some terminology.
For a reflective subuniverse $L$, a type $X$ is \dfn{$L$-separated}
if $x = y$ is $L$-local for every $x, y : X$.

\begin{thm}[{\cite[Theorem~2.25]{CORS}}]
For any reflective subuniverse $L$ of $\UU$, the subuniverse of
$L$-separated types in $\UU$ is again reflective.
\end{thm}

\begin{proof}
Consider the type of $L$-equivalences in $\UU$, namely
$I \defeq \Sigma_{X : \UU} \Sigma_{Y : \UU} \Sigma_{g : X \to Y} \IsEquiv(Lg)$.
Projection onto the third component $g$ defines a family $f$ of maps.
By~\cite[Lemma~2.9]{CORS}, the maps $\eta : X \to LX$ are $L$-equivalences,
and so the $f$-local types in $\UU$ are exactly the $L$-local types.

Now consider the family $\susp f$ which sends $(X; Y; g; p)$ to
the suspension $\susp g : \susp X \to \susp Y$ of $g$.
By~\cite[Lemma~2.15]{CORS}, a type in $\UU$ is $(\susp f)$-local if
and only if it is $L$-separated.

Let $\UU'$ be a universe containing $\UU$, and therefore the type $I$
of $L$-equivalences.  By~\cite[Theorem~2.18]{RSS}, the subuniverse of
$(\susp f)$-local types in $\UU'$ is reflective.  By
\cref{pr:restrict}, it suffices to show that for each $X : \UU$,
$L_{\susp f} X$ is small.  Each $\susp g$ is $(-1)$-connected, so
by~\cite[Theorem~3.12]{CORS} the maps $\eta' : X \to L_{\susp f} X$
are $(-1)$-connected.  Therefore, by Rijke's join construction (the
$n=0$ case of \cref{pr:smallness}), it suffices to show that $L_{\susp
f} X$ is 1-locally small when $X : \UU$, i.e. that $x' = y'$ is small
for $x', y' : L_{\susp f} X$.  Since this is a proposition and $\eta'$
is $(-1)$-connected, it is enough to prove that $\eta'(x) = \eta'(y)$ is
small for $x, y: X$.  But by~\cite[Proposition~2.26]{CORS} (which does
not use Theorem~2.25), $(\eta'(x) = \eta'(y)) \simeq L(x = y)$, which
is in $\UU$.
\end{proof}

This proof uses many of the same ingredients as the proof in~\cite{CORS},
but is a bit more conceptual.

\section{Replacing a family with a single map}\label{se:single-map}

In homotopy type theory, an accessible reflective subuniverse can be
presented as a localization with respect to a family of maps indexed by a type.
Here we investigate when the family can be replaced by a single map.
We work internally in HoTT in \cref{ss:inHoTT}, and work in the
simplicial model in \cref{ss:sSet}.

\subsection{In homotopy type theory}\label{ss:inHoTT}

We begin with a reduction to a family indexed by a 0-type.

\begin{lem}\label{le:accessible-pullback}
Let $f : \Pi_{i : I} A_i \to B_i$ be a family of maps indexed by a type $I$,
and let $g : J \to I$ be $(-1)$-connected.
Then a type $X$ is $f$-local if and only if it is $(f \circ g)$-local,
where 
\[
(f \circ g)(j) \defeq f_{g(j)} : A_{g(j)} \to B_{g(j)} .
\]
\end{lem}

\begin{proof}
It is clear that if $X$ is $f$-local, then it is $(f \circ g)$-local.
(This direction does not use that $g$ is $(-1)$-connected.)

Conversely, suppose that $X$ is $(f \circ g)$-local.
Let $i : I$.
We would like to prove that $(f_i)^* : X^{B_i} \to X^{A_i}$ is an equivalence.
This is a proposition, so we may assume we have $j : J$ with $g(j) = i$.
Since $f_{g(j)}^*$ is an equivalence by assumption, we are done.
\end{proof}

\begin{defn}
The axiom \dfn{sets cover} says that for every type $X$, there merely exists a 0-type $S$
and a $(-1)$-connected map $S \to X$.
For a universe $\UU$, we say that \dfn{sets cover in $\UU$} if for every $X : \UU$,
there merely exists a 0-type $S : \UU$ and a $(-1)$-connected map $S \to X$.
\end{defn}

We show in \cref{ss:sSet} that this holds in the simplicial model.

\begin{prop}
Assume that sets cover in $\UU$.
Then any accessible reflective subuniverse $L$ of $\UU$ can merely be presented
as localization with respect to a family in $\UU$ indexed by a set.
\end{prop}

\begin{proof}
Suppose $L$ is presented as localization with respect to a family $f$ indexed by a type $I : \UU$.
Since we are proving a proposition, we can
choose a $(-1)$-connected map $g : S \to I$ from a set $S$ in $\UU$.
Then $L$ can be presented as localization with respect to the family $f \circ g$
indexed by $S$, using \cref{le:accessible-pullback}.
\end{proof}

Next we consider reducing from a set-indexed family to a single map.
Given a family $f : \Pi_{i : I} A_i \to B_i$ of maps indexed by a type $I$,
write $\tot(f) : (\Sigma_i A_i) \to (\Sigma_i B_i)$ for the induced map
on $\Sigma$-types.
While every $f$-local type $X$ is also $\tot(f)$-local, the converse is
only true under some assumptions.

\begin{prop}\label{pr:inhab}
Let $f : \Pi_{i:I} A_i \to B_i$ be a family of maps.
Then every $f$-local type is $\tot(f)$-local.
If $I$ is a set with decidable equality and $X$ is merely inhabited,
then $X$ is $f$-local if and only if $X$ is $\tot(f)$-local.
\end{prop}

\begin{proof}
Suppose $X$ is $f$-local.  To show that $X$ is $\tot(f)$-local, we must
show that the natural map
\[
  ((\sum_i A_i) \to X) \longleftarrow ((\sum_i B_i) \to X)
\]
is an equivalence.  This map is equivalent to the natural map
\begin{equation}\label{eq:totflocal}
  (\prod_i X^{A_i}) \longleftarrow (\prod_i X^{B_i}) ,
\end{equation}
which is an equivalence since each component is.

Now assume that $I$ is a set with decidable equality and
that $X$ is merely inhabited and $\tot(f)$-local.
Then \eqref{eq:totflocal} is an equivalence.
Let $j : I$.
We must show that the map 
\begin{equation}\label{eq:flocal}
X^{A_j} \leftarrow X^{B_j}
\end{equation}
is an equivalence.
Since this is a proposition, we can assume that we have $x_0 : X$.
We proceed by showing that \eqref{eq:flocal} is a retract of \eqref{eq:totflocal}.
There are natural projection maps from \eqref{eq:totflocal} to \eqref{eq:flocal}
given by evaluation at $j$.
Under our assumptions, these have sections.
For example, we define $X^{A_j} \to \prod_i X^{A_i}$ by sending $f$ to the
function sending $i$ to $f$ when $i = j$ and to the constant map at $x_0$ otherwise.
We define a similar map from $X^{B_j}$.
These are sections of the evaluation maps and make the necessary square commute,
showing that \eqref{eq:flocal} is a retract of \eqref{eq:totflocal}.
\end{proof}

\begin{prop}
Let $f : \Pi_{i:I} A_i \to B_i$ be a family of maps indexed by a set $I$
such that the empty type is $f$-local.
Assuming the law of excluded middle for propositions (LEM), the $f$-local types and
the $\tot(f)$-local types agree.
\end{prop}

Note that if $A_i$ and $B_i$ are merely inhabited for each $i$, then the
empty type is $f$-local.

\begin{proof}
By \cref{pr:inhab}, every $f$-local type is $\tot(f)$-local.
We must prove the converse.
Suppose that $X$ is $\tot(f)$-local.
By LEM applied to $\|X\|_{-1}$, we know that either $X$ is merely inhabited or 
we have $\|X\|_{-1} \to \emptyset$.
If $X$ is merely inhabited, then \cref{pr:inhab} shows that $X$ is $f$-local.
(Since we are assuming LEM, $I$ has decidable equality.)

On the other hand, if we have $\|X\|_{-1} \to \emptyset$, then we have
$X \to \emptyset$ which implies that $X = \emptyset$.
Therefore, $X$ is $f$-local, as required.
\end{proof}

Combining these propositions gives:

\begin{thm}\label{th:single-map}
If sets cover in $\UU$, the law of excluded middle holds and $L$ is an accessible
reflective subuniverse such that the empty type is $L$-local,
then $L$ can merely be presented as localization with respect to a single map. \qed
\end{thm}

\begin{rem}\label{re:single-map}
By \cref{re:local-types},
any localization $L$ produced by \cref{th:non-accessible} with $n \geq 0$
has the property that $\emptyset$ (and all propositions) are $L$-local.
So, if sets cover in $\UU$, the law of excluded middle holds
and $L$ is accessible,
then there merely exists a presentation using a single map.
\end{rem}

\subsection{Properties of the simplicial model}\label{ss:sSet}

In this section, we show that the simplicial model~\cite{KL3} satisfies a strong
form of the axiom of choice, which implies in particular that sets cover
and that the law of excluded middle holds.
These results are known to others.
The reference~\cite{KL2} proves that the law of excluded middle holds in
the simplicial model, but we know of no references for the other facts,
so we include the proofs here to be self-contained.

Given a fibration $f : X \to \Gamma$ of simplicial sets and a vertex $v$ of $\Gamma$,
we write $v : 1 \to \Gamma$ for the vertex inclusion and
$v^* f$ for the fibre $f^{-1}(v)$ of $f$ over $v$.
We say that a fibration or fibrant object is $k$-connected or $k$-truncated if it
satisfies the interpretation of the homotopy type theory definition given in
\cref{de:connected}.

\begin{lem}\label{le:agree1}
A fibrant simplicial set $X$ is $(-1)$-truncated (resp.\ $(-1)$-connected)
if and only if it is empty or contractible (resp.\ non-empty).
\end{lem}

\begin{proof}
It is straightforward to check that in the empty context the definitions
specialize to ``empty or contractible'' and ``non-empty,'' respectively.
\end{proof}

\begin{lem}\label{le:agree}
Let $f : X \to \Gamma$ be a fibration between fibrant simplicial sets.
If $f$ is $(-1)$-truncated (resp.\ $(-1)$-connected),
then for each vertex $v$ of $\Gamma$,
the fibre $v^* f$ is empty or contractible (resp.\ non-empty).
\end{lem}

\begin{proof}
This follows immediately from \cref{le:agree1}, since being
$(-1)$-truncated (resp.\ $(-1)$-connected) is stable under pullback.
\end{proof}

\begin{prop}\label{pr:empty}
Let $p : P \to \Gamma$ be a $(-1)$-truncated fibration between fibrant simplicial sets.
Then $p$ has a section if and only if the fibre $v^* p$ is
non-empty for each vertex $v$ of $\Gamma$.
\end{prop}

\begin{proof}
By \cref{le:agree}, the fibres $v^* p$ of $p$ are empty or contractible.
Suppose that $v^* p$ is non-empty for each vertex $v$.
Then these fibres must be contractible.
This implies that $p$ is a trivial fibration%
\footnote{Citations for the fact that $p$ is a trivial fibration
are~\cite[Proposition 8.23]{J} and~\cite[Lemma~5.4.16 in v2]{RV}.}
and therefore has a section.
The converse is trivial.
\end{proof}

This lets us prove the converse of \cref{le:agree}.

\begin{prop}\label{pr:agree}
Let $f : X \to \Gamma$ be a fibration between fibrant simplicial sets.
Then $f$ is $(-1)$-truncated (resp.\ $(-1)$-connected)
if and only if for each vertex $v$ of $\Gamma$,
the fibre $v^* f$ is empty or contractible (resp.\ non-empty).
\end{prop}

\begin{proof}
One direction is given by \cref{le:agree}.
For the converse, suppose that each $v^* f$ is non-empty.
The condition that $f$ be a $(-1)$-connected map says that a
certain fibration $C \to \Gamma$ has a section, where $C$ is the
interpretation of $\Sigma_y \|\fib_f (y)\|_{-1}$.
Since $C \to \Gamma$ is a family of mere propositions, the map
$C \to \Gamma$ is $(-1)$-truncated.
On the other hand, the assumption that each $v^* f$ is non-empty
tells us that the fibres of $C \to \Gamma$ are non-empty,
since the interpretation of $C$ commutes with pullback to each vertex.
So it follows from \cref{pr:empty} that $C \to \Gamma$ has a section.

A similar argument works for $(-1)$-truncated maps.
\end{proof}

We will now see how these propositions let us reduce many questions to the empty context.

\medskip

For $-1 \leq n \leq \infty$, let $\AC_n$ be the statement that
for every $0$-type $X$ in some context $\Gamma$, every family of
merely inhabited $n$-types over $X$ merely has a section.
When $n = \infty$, there is no truncation condition on the family.
This is $\AC_{n,-1}$ from~\cite[Exercise~7.8]{Uni}, but phrased
in a way that avoids universes.
The following result is mentioned in that exercise, without proof.

\begin{thm}\label{th:AC}
The simplicial model satisfies $\AC_{\infty}$.
More precisely, for every $0$-truncated fibration $p : X \to \Gamma$
and every $(-1)$-connected fibration $e : Y \to X$,
the $(-1)$-truncation of the fibration $\Pi_p e$ in the slice over $\Gamma$
has a section.
\end{thm}

\begin{proof}
Let $p$ and $e$ be as in the statement.
We need to show that the $(-1)$-truncation of the fibration $\Pi_p e$
has a section.
In other words, we need to show that $\Pi_p e$ is a $(-1)$-connected map.
By \cref{pr:agree}, it is enough to show that for each vertex $v$
of $\Gamma$, the preimage $v^*(\Pi_p e)$ is non-empty.
Since $\Pi$ commutes with pullbacks, this preimage is equivalent
to $\Pi_{v^* p} v^* e$, which puts us in the empty context.
Here $v^* e$ denotes the functorial action of pullback, giving
a $(-1)$-connected fibration over a homotopically
discrete simplicial set $v^* p$.
Using the external axiom of choice, choose one vertex of each component
to obtain a trivial cofibration $S \to v^* p$ from a discrete simplicial set.
Using choice again, we can lift this map as shown, since $v^* e$ is $(-1)$-connected:
\[
  \begin{tikzcd}
    S \arrow[r,dashed] \arrow[d,swap,] & v^*(e \circ p) \arrow[d,"v^* e"] \\
    v^* p \arrow[r,equal] & v^* p .
  \end{tikzcd}
\]
Finally, we use the lifting property of trivial cofibrations and fibrations
to obtain the desired section of $v^* e$.
\end{proof}

As another example, we give an alternate proof of the following result.

\begin{prop}[{\cite{KL3}}]\label{pr:LEM}
The law of excluded middle holds in the simplicial model.
That is, for every $(-1)$-truncated fibration $q : A \to \Gamma$ of simplicial sets,
the fibration $p : P \to \Gamma$ corresponding to $\LEM(A) \defeq A + \lnot A$ has a section.
\end{prop}

\begin{proof}
Note that $p$ is $(-1)$-truncated, since $\LEM(A)$ is a mere proposition when $A$ is.
So, by \cref{pr:empty},
it is enough to show that $v^* p$ is non-empty for each vertex $v$ of $\Gamma$.
Since LEM commutes with pullback, $v^* p$ is equivalent to $\LEM(v^* q)$.
In other words, we have reduced the problem to the empty context.
But for any fibrant simplicial set $X$, $X + \lnot X$ has a global element,
since $X$ is either non-empty or empty.
\end{proof}

The previous result also follows from Diaconescu's Theorem~\cite[Theorem~10.1.14]{Uni}
and \cref{th:AC}, since $\AC_{\infty}$ clearly implies $\AC_0$.

\begin{cor}\label{co:sets-cover}
For any universe $\UU$ in the simplicial model, sets cover in $\UU$.
\end{cor}

It follows that \cref{re:single-map} applies in the simplicial model.

\begin{proof}
\cite[Exercise~7.9]{Uni} says that $\AC_\infty$ implies that sets cover,
and this has been formalized by Jarl Flaten in~\cite{HoTT}.
The set that covers $X$ can be taken to be $\pi_0(X)$, so it lives
in the same universe as $X$.
So the claim follows from \cref{th:AC}.
\end{proof}

Alternatively, one can use the methods above to show that it
suffices to check this in the empty context.
Then, using choice, one can also find a $(-1)$-connected map
$\pi_0(X) \to X$.
\cref{co:sets-cover} also follows from \cite[Theorem~3.1]{Shu'}, which
relates this to the external notion.
\cite{Shu'} also points out that the external notion fails in some slices
of simplicial sets.  We give a simpler argument here.

\begin{prop}
Let $f : Y \to X$ be a map of simplicial sets with $Y$ non-empty and $X$ $1$-connected.
If there exists a set $g : S \to X$ in the slice over $X$ and a $(-1)$-connected map $e$ from
$g$ to $f$, then $f$ has a section.
Taking $X$ to be any non-trivial $1$-connected space and $f$ to be the inclusion of a point,
we see that the external notion of sets cover fails in the slice over $X$.
\end{prop}

\begin{proof}
Let $f$, $g$ and $e$ be as in the statement.
Since $Y$ is non-empty and $e$ is $(-1)$-connected, $S$ is non-empty.
Since $X$ is $1$-connected, the family $g$ of sets over $X$ is trivial (and non-empty),
and therefore has a section $s$.
Then $f \circ e \circ s = g \circ s = \id_X$, so $e \circ s$ is a section of $f$.
\end{proof}

\medskip

The theme of this section is that many questions can be reduced
to the empty context.
This is fundamentally because of \cref{pr:agree}, which says that $(-1)$-connected
maps are determined by their pullbacks to the terminal object,
and corresponds to the model being ``well-pointed'' in some sense.
Note that not every question can be checked in the empty context.
For example, for any fibrant simplicial set $X$, $\LEM(X)$ holds,
but it does not hold for families.

\section{The relationship to past work}\label{se:pastwork}

\subsection{The local objects}\label{ss:local-objects}

The paper~\cite{CSS} works in the category of simplicial sets.
The local objects in~\cite[Theorem~6.4]{CSS} are the 1-types $X$ such
that for every (small) infinite cardinal $\kappa$ and every $x \in X$,
\[
  \Hom(\Z^{\kappa}/\Z^{<\kappa},\, \pi_1(X,x)) = 1 .
\]
Here, $\Z^{\kappa}$ is the abelian group of all functions $f : \kappa \to \Z$,
and $\Z^{<\kappa}$ is the subgroup consisting of those functions $f$ whose
support has cardinality smaller than $\kappa$.
In homotopy type theory, we can choose a family of groups that gives the
same condition when interpreted in simplicial sets,
where the law of excluded middle holds for propositions~(see~\cite{KL2} and \cref{pr:LEM}).
For example, one could index the family by 0-types $K$ in $\UU$ which are
``infinite'' in some sense, e.g., in that they merely have a self-map that
is injective but not surjective.
For each such $K$, one considers the group $\Z^K/\Z^{<K}$,
where $\Z^K \defeq (K \to \Z)$ and $\Z^{<K}$ could be defined to be the subgroup
generated by those $f$ satisfying
\[
\lnot \, \bigg(\sum_{g \colon \textup{supp}(f) \to K} \, \mathsf{is}\text{-}(-1)\text{-}\mathsf{connected}(g)\bigg).
\]
Variations are possible and we make no claim that these definitions are the
most practical for future work.  We only claim that in the simplicial model,
these produce the same family as used in~\cite{CSS} (with repetition, which is
harmless).

More generally, one can consider a class $S$ of epimorphisms of small groups
and define the $S$-local objects to be the 1-types $X$ such
that for every $f : A \to C$ in $S$ and every $x \in X$,
the natural map
\[
  \Hom(A,\, \pi_1(X,x)) \leftarrow \Hom(C,\, \pi_1(X,x))
\]
is a bijection.
(Taking the class of maps $\Z^{\kappa}/\Z^{<\kappa} \to 1$ reproduces the previous case.)
We can characterize these local objects in the following way, using that
for a 1-type $X$, $\Hom(A, \pi_1(X,x))$ is equivalent to the type $BA \pto (X,x)$
of pointed maps from the classifying space $BA \equiv K(A,1)$ to the pointed
type $(X, x)$:
\begin{align*}
   \{ \text{$S$-local objects} \}
&= \{ \text{1-types } X \mid \prod_{f:A \to C} \, \prod_{x:X} \, (BA \to_* (X,x)) \leftarrow (BC \to_* (X,x)) \text{ is an equiv.}\} \\
&= \{ \text{1-types }X \mid \prod_{f:A \to C} \, (BA \to X) \leftarrow (BC \to X) \text{ is an equiv.}\} \\
&= \{ \text{1-types } X \mid X \text{ is local w.r.t. $\{Bf : BA \to BC \mid f \in S\}$}\}.
\end{align*}
Two of the products are over $f \in S$, and the first two maps are induced by
$Bf : BA \pto BC$.
The second equality is~\cite[Lemma~3.3]{CORS}.
(We are using type-theoretic notation here, but this can all be interpreted classically.)
When $f$ is a surjection of groups, $Bf$ is a $0$-connected map,
so \cref{th:non-accessible} gives a localization onto the types described
in the last displayed line.
(Note that the internal and external homs agree for simplicial sets,
so there is no need to distinguish between ``internally'' and ``externally'' local objects.)

We next discuss the interpretation of this in more detail.

\subsection{The localization}

We will compare the results of this paper with~\cite{CSS} by comparing
what they give in the $\infty$-topos of spaces, i.e., in the simplicial nerve
$N(\Kan)$ of the simplicial category $\Kan$ of Kan complexes.
(This is also called the homotopy coherent nerve.)

The paper~\cite{CSS} was written before the theory of $\infty$-categories
was well-established.
For a class $S$ of surjections of small groups,
they construct a \dfn{homotopy idempotent functor}
$E$ whose local objects are the $S$-local spaces.
More precisely, $E$ is a (strict) functor $\sSet \to \sSet$ which sends
weak equivalences to weak equivalences and is equipped with a (strict) natural
transformation $\eta_X : X \to EX$ such that, for each $X$, $\eta_{EX}$ and $E \eta_X$
are weak equivalences $EX \to E E X$.%
\footnote{%
It is because of this strictness that we compare our results in the $\infty$-topos
of spaces.
It is not clear whether a general localization of the $\infty$-topos of
spaces can be ``strictified'' into a homotopy coherent functor.
Most results along these lines deal with accessible localizations.
}
(They note that it follows that $\eta_{EX}$ and $E \eta_X$ are homotopic,
and that one can assume without loss of generality that $EX$ is a Kan complex
for every simplicial set $X$.)

By~\cite[Corollary~6.5]{RSS'}, one can replace $E$ by a weakly equivalent
\emph{simplicial} functor $\sSet \to \sSet$.
One can then apply the simplicial nerve functor $N$ to obtain an ($\infty$-)functor $E'$
on the $\infty$-topos $N(\Kan)$ of spaces, which gives a
localization in Lurie's sense (see~\cite[Proposition 5.2.7.4]{Lur}).
The essential image of $E'$ consists of the $S$-local Kan complexes.

Now we consider how \cref{th:non-accessible} gives rise to a localization
in Lurie's sense.
There are several subtleties to deal with, and we only give an outline.
Let $\beta < \alpha$ be (strongly) inaccessible cardinals.
By~\cite{KL3}, there is a model of homotopy type theory based on the
category of $\alpha$-small simplicial sets (those $X$ such that each $X_n$
has cardinality less than $\alpha$) and which contains an internal universe
$U_{\beta}$ corresponding to the $\beta$-small simplicial sets.
We regard the $\beta$-small simplicial sets as the usual ``small'' simplicial sets,
and the larger simplicial sets as a technical tool.
(\cite{CSS} also makes use of such larger simplicial sets.)
We take $U_{\beta}$ as our interpretation of the universe $\UU$ appearing
in \cref{th:non-accessible}.

\cref{th:non-accessible} produces a reflective subuniverse of $\UU$.
While we expect that one could interpret this directly as a localization
of the $\infty$-topos of $\beta$-small Kan complexes, this has not been
written down explicitly.
The Appendix of \cite{RSS} instead gives the interpretation of a ``judgmental reflective subuniverse''
that acts on \emph{all} types, not just those in some universe.
The reflective subuniverse $L_{\bar{f}}$ that appears in the proof of
\cref{th:non-accessible} is polymorphic in this sense, as it is accessible.
Therefore, \cite[Theorem~A.12]{RSS} applies and tells us that we get
a reflective subfibration of the $\infty$-topos of $\alpha$-small Kan complexes,
which in particular gives a localization in Lurie's sense.
The localization produced by \cref{th:non-accessible} is then interpreted
as the restriction of this localization to the sub-$\infty$-topos of Kan complexes
which are $\beta$-small (up to equivalence).

As a localization is determined by its essential image, the discussion
of the previous section shows that this localization agrees with the
localization coming from the work of~\cite{CSS}.

\subsection{Accessibility}

By \cref{th:single-map}, \cref{re:single-map} and the results of \cref{ss:sSet},
if the localization produced by \cref{th:non-accessible} is accessible in homotopy type theory,
then it can merely be presented as localization with respect to a single map.
In particular, it would follow that the local spaces described in
\cref{ss:local-objects} are $f$-local for a single map $f$ between
simplicial sets.
This is what~\cite{CSS} show is not provable from ZFC (assuming that ZFC
is consistent), and so it follows that the localization in homotopy type theory
also cannot be proven to be accessible.

Note also that \cite[Theorem~A.20]{RSS} shows that if $L$ is an accessible
reflective subuniverse in homotopy type theory,
then the corresponding reflective subcategory in an $\infty$-topos model is accessible.
This does not itself guarantee that it can be presented using a single map,
but the arguments of \cref{se:single-map} can be carried out directly in
the simplicial model to show that this is the case.

\section{Formalization}\label{se:formalization}

The main results of this paper have been formalized using the
Coq HoTT library~\cite{HoTT} and are available at~\cite{Chr}.
Specifically, at the time of writing, the following results have been formalized:
\cref{pr:smallness},
\cref{le:truncated-type},
\cref{le:truncated-family},
\cref{le:dual-smallness},
\cref{th:characterize-smallness} (both proofs),
\cref{co:dual-smallness},
\cref{re:issmall-inhabited-issmall},
\cref{pr:restrict},
and \cref{th:non-accessible}.
The \texttt{README.md} file at~\cite{Chr} gives the mapping
between the results here and the formalized results and will be updated
as more results are formalized.
After further polishing, the results will be pushed to the main library.

The formalization of the above results takes approximately 500 lines.
In order to do these formalizations, a number of other known results
needed to be formalized.
For example, we prove several results about connected maps,
enough to show that if $f$ is a surjection between groups, then
$Bf$ is a $0$-connected map.
(See \cref{ex:surjective} and~\cite[Corollary~8.8.5]{Uni}.)
We also prove a special case of~\cite[Theorem~3.12]{CORS},
namely that if $O$ is a modality and $L_f$ is the localization
with respect to a family with each $f_i$ $O$-connected,
then each $\eta : X \to L_f X$ is $O$-connected.
This was challenging because~\cite[Lemma~1.44]{RSS} has not been
formalized, so we had to also prove various results about orthogonal
factorization systems.
The results described in this paragraph required about 600 additional lines.

In addition, the $n=0$ case of \cref{pr:smallness},
which was proved informally by Rijke~\cite{R},
is assumed as an axiom in the formalization.
It has been formalized in Arend by Valery Isaev~\cite{Isa}
(see the file \texttt{Homotopy/Image.ard}).

\vspace*{1.5ex}

\end{document}